\documentclass[12pt,oneside,reqno]{amsart}
\usepackage{amsmath,amssymb, amsthm,mathtools}
\usepackage[margin=1.5in]{geometry}
\usepackage[parfill]{parskip}
\usepackage{enumitem}
\usepackage{tikz}
\usepackage{tikz-cd} 
\usepackage{dsfont}
\usepackage{mathrsfs}
\usepackage{bbm}
\usepackage{dirtytalk}
\usepackage{bm}
\usepackage{microtype}
\usepackage{caption}
\usepackage{subcaption}
\usepackage{array}
\usepackage{longtable}

\providecommand{\N}{\mathbb{N}}
\providecommand{\F}{\mathbb{F}}

\providecommand{\Z}{\mathbb{Z}}

\providecommand{\leqst}{\leq_{\mathrm{st}}} 
\providecommand{\Hom}{\mathrm{Hom}}

\newcommand{\paren}[1]{\left( #1 \right)}
\newcommand{\brac}[1]{\left[ #1 \right]}

\newcommand{\abs}[1]{\left\vert#1\right\vert}

\newcommand{\set}[1]{\left\{#1\right\}}

\DeclareMathOperator{\im}{im}

\newtheorem{Theorem}{Theorem}
\newtheorem{Lemma}[Theorem]{Lemma}
\newtheorem{Definition}[Theorem]{Definition}
\newtheorem{Proposition}[Theorem]{Proposition}

\newtheorem{Corollary}[Theorem]{Corollary}

\newcounter{cnstcnt}

\usepackage{graphicx}
\begin{document}

\title[Some Properties of the PRCM]{Some Properties of the Plaquette Random-Cluster Model}
\author{Paul Duncan}
\email{paul.duncan@mail.huji.ac.il}
\address{Einstein Institute of Mathematics, Hebrew University of Jerusalem, Jerusalem 91904, Israel}
\author{Benjamin Schweinhart}
\email{bschwei@gmu.edu}
\address{Department of Mathematical Sciences, George Mason University, Fairfax, VA 22030, USA}

\begin{abstract}
We show that the $i$-dimensional plaquette random-cluster model with coefficients in $\Z_q$ is dual to a $(d-i)$-dimensional plaquette random cluster model. In addition, we explore boundary conditions, infinite volume limits, and uniqueness for these models. For previously known results, we provide new proofs that rely more on the tools of algebraic topology. 
\end{abstract}
\maketitle

\section{Introduction}

The plaquette random-cluster model (PRCM) with coefficients in $\Z_q$ for $q \in \N$ is the random $i$-dimensional subcomplex $P$ of a cubical complex $X$ so that 
$$\mu_{X,p,q,i}\paren{P} \propto p^{\abs{P}}\paren{1-p}^{\abs{X^{\paren{i}}} - \abs{P}}\abs{H^{i-1}\paren{P;\;\Z_q}}$$
where $\abs{P}$ and $\abs{X^{\paren{i}}}$ denote the number of $i$-plaquettes of $P$ and $X,$ respectively, and $H^{i-1}\paren{P;\;\Z_q}$ is the cohomology of $P$ with coefficients in $\Z_q.$ This definition was first suggested in~\cite{duncan2022topological} (rather, an equivalent one defined in terms of homology; see Corollary~\ref{cor:UCTC} below), and the details were worked out independently by~\cite{duncan2023sharp} and \cite{shklarov2023edwards}. The PRCM is motivated by its coupling with $(i-1)$-dimensional $q$-state Potts lattice gauge theory which assigns spins to the $(i-1)$ cells of $X$; the Wilson loop expectation for an $(i-1)$-boundary $\gamma$ equals the probability that $\brac{\gamma}$ is null-homologous when homology coefficients are taken in $\Z_q.$ When $q$ is prime this PRCM coincides with the plaquette random-cluster model with coefficients in the field $\F_q.$ The latter was first introduced by~\cite{hiraoka2016tutte}, which focused on a mean-field version of the model. All results presented here also hold for the PRCM with coefficients in a field, with minor modifications.

Graphical representations have proven to be a useful tool in the study of lattice spin models such as the Potts model. For the PRCM --- a cellular representation of Potts lattice gauge theory --- to play the same role, its basic properties must be elucidated. The methods of~\cite{duncan2023sharp} relied on a technical shortcut to obtain basic results about the codimension one PRCM on the way to prove a sharp phase transition for Wilson loop expectations in $(d-2)$-dimensional Potts lattice gauge theory on $\Z^d.$ Specifically, $(d-1)$-dimensional PRCM on $\Z^d$ with coefficients in an abelian group is equivalent to a PRCM with coefficients in a field, which is in turn dual to the the classical $1$-dimensional random-cluster model (RCM). Thus, results for the RCM concerning boundary conditions, positivity, and infinite volume limits can be translated to corresponding statements for the codimension one PRCM. We cannot rely on this logic more generally.  The purpose of this paper is to prove corresponding results for general values of $i,$ including the special case of the self-dual $2$-dimensional PRCM on $\Z^4$ which is coupled with $1$-dimensional Potts lattice gauge theory. We hope that this will be helpful for researchers tackling this particularly interesting case.

One of our main goals is to show that an $i$-dimensional PRCM on $\Z^d$ on $\Z_q$ with parameter $p$ is dual to a $(d-i)$-dimensional PRCM on $\Z^d$ with coefficients in $\Z_q$ and parameter $p^*\paren{p,q}$ where
$$  p^*(p,q)= \frac{\paren{1-p}q}{\paren{1-p}q + p}\,.$$
Towards that end, we study boundary conditions and infinite volume measures for the PRCM and prove a number of results about them. Some of these latter results were also shown by~\cite{shklarov2023edwards}, but our proofs are shorter and employ different, more geometric arguments.

Before proceeding, we give a definition of the PRCM on a box with boundary conditions. We will explain how this generalizes the standard construction for the RCM and provide more intuition in Section~\ref{sec:infvolume0}. Fix $i,$ let $r$ be a rectangular box in $\Z^d$ and let $\xi$ be a collection of $i$-plaquettes. Denote by $P_{\xi}$ the union of $P, \xi \cap \paren{\Z^d \setminus r},$ and the $(i-1)$-skeleton of $\Z^d,$ and write $\phi$ for the inclusion map from $P$ into $P_{\xi}.$ Then the PRCM on $r$ with boundary conditions $\xi$ is the measure $\mu_{r,p,q,i}^{\xi}$ is defined by
 \[\mu_{r,p,q,i}^{\xi}\paren{P} \propto p^{\abs{P}}\paren{1-p}^{\abs{X^{\paren{i}}} - \abs{P}}\abs{\im\phi^*}\,.\]
While this does not include the important case of periodic boundary conditions, the corresponding results about that case follow by very similar arguments (see~\cite{duncan2022topological}). 

We give an informal description of our results. First, we show that if the PRCMs on a sequence of nested boxes $r_1\subset r_2\subset r_3\ldots$ constructed with boundary conditions $\xi$ converges to an infinite volume measure, then the correctly chosen dual PRCMs converge to a dual infinite volume measure. There is some subtlety in this; the dual PRCM is ``wired at infinity'' in a sense made precise using Borel--Moore homology. The duality theorem is proven in two steps: by establishing a finite volume analogue in Theorem~\ref{thm:mostgeneralduality} and then extending it to infinite volume measures in Theorem~\ref{thm:limduality}. We also establish a number of technical results about the PRCM with boundary conditions.  Proposition~\ref{prop:largeboundary} states that for sufficiently large boxes $\hat{r}\supset r,$ 
this measure coincides with the free PRCM $\hat{P}$ on $\hat{r}$ conditioned to agree with $\xi$ on $\hat{r} \setminus r.$

Next, we extend a theorem of Grimmett on the classical RCM~\cite{grimmett1995stochastic} to show that there is a unique infinite volume PRCM for generic values of $p$ and fixed values of $i,d,$ and $q$ (Theorem~\ref{thm:uniqueness}). Finally, we show that finite volume (Theorem~\ref{thm:FKGgeneral}) infinite volume PRCMs (Corollary~\ref{cor:FKGinfinite}) are positively associated using the Mayer--Vietoris sequence for homology.

\section{Background and Definitions}
In this paper we consider  random subcomplexes of the natural cubical complex structure on the integer lattice $\Z^d.$ The $k$-dimensional cells of this complex are exactly the translates and rotations of the unit cube $\brac{0,1}^k$ which have integer corner points. For a subcomplex $X,$ we write $X^{\paren{k}}$ for the union of the cells of $X$ of dimension at most $k.$ If the highest dimensional cell of $X$ is $k,$ then write $\abs{X}$ for the number of $k$-cells of $X.$ We say that $X$ is an $i$-dimensional percolation subcomplex of $Y \subset \Z^d$ if $Y^{\paren{i-1}} \subset X \subset Y^{\paren{i}}.$ In particular, any subset of the $i$-dimensional cells is permitted. We most often consider rectangular subsets of $\Z^d$ of the form $r = \prod_{k=1}^d\brac{a_k,b_k},$ which we call boxes. We will also write $r$ for the union of its cells of dimension at most $(i-1)$ and the $i$-cells which intersect its interior (alternatively, we exclude the $i$-cells contained in its boundary). For the union of all cells of dimension at most $i$ in $r,$ we instead write $\overline{r}.$ Sometimes it will also be convenient to work with the cube of side length $2n,$ which we write as $\Lambda_n \coloneqq \brac{-n,n}^d.$ In this paper, we assume familiarity with homology and cohomology. For a non-specialist introduction, see the first appendix to~\cite{duncan2023sharp}. 

The $i$-dimensional plaquette random-cluster model (PRCM) of a finite subcomplex $X \subset \Z^d$ with parameters $p \in \brac{0,1},$ $q \in \N+2$ is defined as follows:

\begin{Definition}\label{def:prcm}
    Let $\mu_{X,p,q,i}$ be the measure on percolation subcomplexes $P \subset X$ given by
    \[\mu_{X,p,q,i} \propto p^{\abs{P}}\paren{1-p}^{\abs{X^{\paren{i}}} - \abs{P}}\abs{H^{i-1}\paren{P;\;\Z_q}}\,.\]
\end{Definition}

The PRCM can equivalently be defined in terms of homology rather than cohomology as a consequence of the universal coefficient theorem. We recall a formulation given in~\cite{duncan2023sharp}.

\begin{Proposition}\label{prop:UCTC}
If $H_{j-2}\paren{P;\;\Z_q}$ vanishes (or, more generally, is a free $\Z_q$-module) then
\begin{equation*}
\label{eq:UCTC}
H^{j-1}\paren{P;\;\Z_q}\cong H_{j-1}\paren{P;\;\Z_q}\,.
\end{equation*}
\end{Proposition}

A percolation subcomplex $P$ of a box $r$ satisfies the condition $H_{j-2}\paren{P;\;\Z_q}$ , leading to the following corollary.
\begin{Corollary}\label{cor:UCTC}
    Define
    \[\hat{\mu}_{X,p,q,i} \propto p^{\abs{P}}\paren{1-p}^{\abs{X^{\paren{i}}} - \abs{P}}\abs{H_{i-1}\paren{P;\;\Z_q}}\,.\]
    Then for any box $r \subset \Z^d,$
    \[\mu_{r,p,q,i} \,{\buildrel d \over =}\, \hat{\mu}_{r,p,q,i}\,.\]
\end{Corollary}
As a result, we may either work with the homology or cohomology of $P$ as is convenient. The definition involving cohomology is more natural in the context of the coupling of the PRCM with Potts lattice gauge theory (PLGT); $\abs{H^{i-1}\paren{P;\;\Z_q}}$ counts equivalence classes of spin assignments to the $(i-1)$-faces of $X.$ 

Recall that the $(i-1)$-dimensional $q$-state PLGT with inverse temperature parameter $\beta$ on a finite subset $X \subset \Z^d$ is the random element of $C^{i-1}\paren{X;\;\Z_q}$ distributed according to
\[\nu_{X,\beta,q,k}\paren{f} \propto e^{-\beta H\paren{f}}\,,\]
where $H$ is the Hamiltonian defined by
\begin{equation}
    \label{eq:hamiltonian_potts}
    H\paren{f}=-\sum_{\sigma}K\paren{\delta f\paren{\sigma},0}\,.
\end{equation}
Here $K$ is the Kronecker delta function and $\delta$ is the coboundary operator.

The PRCM and PLGT can be coupled in fashion analogous to the Edwards--Sokal coupling of the classical random-cluster model and the Potts model~\cite{edwards1988generalization,swendsen1987nonuniversal}. This was proven for general $q$ independently in~\cite{duncan2023sharp} and~\cite{shklarov2023edwards}.

\begin{Theorem}[\cite{duncan2023sharp,shklarov2023edwards}]\label{prop:couple}
Let $X$ be a finite cubical complex, $q\in\N+1,$ $\beta \in [0,\infty),$ and $p = 1-e^{-\beta}.$ Define a coupling on $C^{i-1}\paren{X}\times \set{0,1}^{X^{\paren{i}}}$ by
\[\kappa\paren{f,P} \propto \prod_{\sigma \in X^{\paren{i}}}\brac{\paren{1-p}I_{\set{\sigma \notin P}} + p I_{\set{\sigma \in P,\delta f\paren{\sigma}=0}}}\,.\]

Then $\kappa$ has the following marginals.
\begin{itemize}
    \item The first marginal is
    $\nu_{X,\beta,q,i-1}.$ 
    \item The second marginal is
    $\mu_{X,p,q,i}.$ 
\end{itemize}
\end{Theorem}

\section{Boundary Conditions}\label{sec:infvolume0}

First, consider the familiar random-cluster model on a graph. A boundary condition on a subgraph $S \subset \Z^d$ induced by some finite vertex set can be thought of as a configuration of edges not contained in $S.$

Let $\xi$ be a set of edges in $\Z^d$ and write $P^{\xi}$ for the set of open edges of $\xi \cap \paren{\Z^d \setminus S}.$ The idea is to define a random-cluster measure on $S$ with the additional edges of $P^{\xi}$ added for the purpose of counting connected components. Of course, $P^{\xi}$ will have infinitely many connected components in general, but finitely many of them are connected to $S.$ 

More precisely, there is a corresponding random-cluster measure on $S$ with boundary condition $\xi$ written as $\mu_{S,p,q,1}^{\xi}\paren{P},$ where the term $\mathbf{b}_0\paren{P}$ counting the number of connected components of $P$ in $S$ is replaced by the number of connected components of $P_{\xi}$ that intersect $S.$ The extremal cases of $\xi$ containing all closed or all open edges are called free and wired boundary conditions respectively, and we write $\mu_{S,p,q,1}^{\mathbf{f}}$ and $\mu_{S,p,q,1}^{\mathbf{w}}$ for the associated measures.

Boundary conditions in the PRCM on $\Z^d$ are defined analogously, in that we want to define a random-cluster model on a subcomplex $X$ with the additional topological information from external plaquettes. Let $\xi$ be a set of plaquettes and recall that $P_{\xi} = P \cup \paren{\xi \cap \paren{\Z^d \setminus X}} \cup \paren{\Z^d}^{\paren{i-1}}.$

\begin{Definition}\label{defn:boundaryconditions}
 Let $\phi : P \to P_{\xi}$ be the inclusion map and let 
 \[\phi^* : H^{i-1}\paren{P_{\xi};\;\Z_q} \to H^{i-1}\paren{P;\;\Z_q}\]
 be the induced map on cohomology. The measure $\mu_{X,p,q,i}^{\xi}$ is defined by
 \[\mu_{X,p,q,i}^{\xi}\paren{P} \propto p^{\abs{P}}\paren{1-p}^{\abs{X^{\paren{i}}} - \abs{P}}\abs{\im\phi^*}\,.\]
 \end{Definition}

Note that taking $i=1$ recovers the definition for the classical RCM. To get a feel for what this definition means, we consider the examples of free and wired boundary conditions. In the former case, $\phi^*$ is surjective and 
 \[\im \phi^* = H^{i-1}\paren{P;\;\Z_q}\]
so the measure (denoted  $\mu_{X,p,q,i}^{\mathbf{f}}$) coincides with the earlier definition of the random-cluster measure on the finite complex $X.$ On the other hand, as long as $X$ does not have nontrivial global homology itself, an element of $H^{i-1}\paren{P_{\xi};\;\Z_q}$ must vanish on $(i-1)$-cycles supported on the boundary of $X$ when we use wired boundary conditions. The wired measure $\mu_{X,p,q,i}^{\mathbf{w}}$ is then the same as the finite volume random-cluster measure on $X$ with the boundary $(i-1)$-cells all identified. Specifically, when $X=r$ is a box in $\Z^d,$ we have that
\[\im\phi^* \cong H^{i-1}\paren{P \cup \partial r;\;\Z_q}\,,\]
a term which will appear again when we discuss Alexander duality. 

An analogue of Corollary~\ref{cor:UCTC} on the equivalence of homological and cohomological perspectives also holds for the PRCM with boundary conditions, which we will prove shortly.

\begin{Lemma}\label{lemma:UCTCboundaryfinite}
    Let $P_1 \subset P_2$ be percolation complexes, and let $\phi^*:H^{i-1}\paren{P_2;\;\Z_q}\rightarrow H^{i-1}\paren{P_1;\;\Z_q}$ and $\phi_*:H^{i-1}\paren{P_1;\;\Z_q}\rightarrow H^{i-1}\paren{P_2;\;\Z_q}$ be the homomorphisms induced by the inclusion $\phi:P_1\hookrightarrow P_2.$ Then 
    \[\abs{\im\phi^*} = \abs{\im\phi_*}\,.\]
    In particular, 
    \[\mu_{X,p,q,i}^{\xi}\paren{P} \propto p^{\abs{P}}\paren{1-p}^{\abs{X^{\paren{i}}} - \abs{P}}\abs{\im\phi^*} = p^{\abs{P}}\paren{1-p}^{\abs{X^{\paren{i}}} - \abs{P}}\abs{\im\phi_*}\,.\]
\end{Lemma}

For finitely supported boundary conditions, there is a straightforward relationship between dual boundary conditions in terms of the definitions that we have already provided. However, the general case is more subtle, and in order to state Theorem~\ref{thm:mostgeneralduality} in full generality, we will also want a notion of boundary conditions that are ``wired at infinity.'' Here it will be more convenient to work with homology rather than cohomology. As motivation, consider an approximation of $\xi$ given by $\xi \cap r$ for a large box $r.$ We will soon see that $\mu_{X,p,q,i}^{\xi} =\mu_{X,p,q,i}^{\xi \cap r}$ for sufficiently large $r.$ This can be thought of as a free approximation, since it is equivalent to setting all sufficiently distant plaquettes to be closed. One could just as easily consider a wired approximation, in which the distant plaquettes are taken to be open. This also converges, but to a possibly different limit. For example, consider the classical random-cluster model in a box with boundary conditions that contain two disjoint infinite paths meeting the boundary of the box at vertices $v$ and $w.$ Clearly, $v$ and $w$ are externally connected in any wired approximation and externally disconnected in any free approximation.

In order to capture the limit of wired approximations, it is then natural to consider cycles that ``pass through infinity'' in some sense. One way to formalize this is using Borel--Moore homology, for which an exposition of the viewpoint we use here can be found in Chapter 3 of~\cite{hughes1996ends}. Recall that Borel--Moore homology of a space $X$ with $i$-cells $X^{\paren{i}}$ can be defined in terms of the locally finite chain groups
\[C_k^{\mathrm{BM}}\paren{X;\;\Z_q} \coloneqq 
\set{\sum_{\sigma \in X^{\paren{k}}} a_{\sigma} \sigma : a_{\sigma} \in \Z_q}\,.\]
The important difference between these and the usual chain groups is that the sum is permitted to have infinitely many nonzero terms. The usual boundary operator can then be extended linearly to obtain 
\[\partial_k^{\mathrm{BM}} : C_k^{\mathrm{BM}}\paren{X;\;\Z_q} \to C_{k-1}^{\mathrm{BM}}\paren{X;\;\Z_q}\,,\]
and then the homology is given by
\[H_k^{\mathrm{BM}}\paren{X;\;\Z_q} \coloneqq \ker \partial_k^{\mathrm{BM}} / \im \partial_{k+1}^{\mathrm{BM}}\,.\] For example, $H_d\paren{\Z^d;\;\Z_q}=0$ but the oriented sum of all $d$-cells of $\Z^d$ is a non-trivial Borel--Moore cycle and $H_d^{\mathrm{BM}}\paren{\Z^d;\;\Z_q}\cong \Z_q.$ Also, $H_0\paren{\Z^d;\;\Z_q}\cong \Z_q$ and $H_0^{\mathrm{BM}}\paren{\Z^d;\;\Z_q}=0.$ Finally, returning to the example with infinite paths extending from $v$ and $w,$ there is a Borel--Moore chain with boundary $v-w,$ namely the (infinite) sum of the edges in the two paths.

\begin{Definition}\label{defn:wiredboundaryconditions}
    As before, denote by $\phi : P \to P_{\xi}$ the inclusion map. Also, let 
    \[\phi_*^{\mathrm{BM}} : H_{i-1}^{\mathrm{BM}}\paren{P;\;\Z_q} \to H_{i-1}^{\mathrm{BM}}\paren{P_{\xi};\;\Z_q}\]
    be the induced map on Borel--Moore homology. Then define
    \[\mu_{X,p,q,i}^{\overline{\xi}}\paren{P} \propto p^{\abs{P}}\paren{1-p}^{\abs{X^{\paren{i}}} - \abs{P}}\abs{\im\paren{\phi_*^{\mathrm{BM}}}}\,.\]
\end{Definition}

This is a slight abuse of notation, since we have not defined $\overline{\xi}$ by itself, but we only write it in the context of this measure. Note that if $\xi$ contains all but finitely many plaquettes of $\Z^d$ then $\mu_{X,p,q,i}^{\overline{\xi}}{\buildrel d \over =} \mu_{X,p,q,i}^{\xi}.$ We will show later (in Proposition~\ref{prop:largeboundary}) boundary conditions of the form $\bar{\xi}$ can be modified so that this is the case. Unlike Definitions~\ref{def:prcm} and~\ref{defn:boundaryconditions}, Definition~\ref{defn:wiredboundaryconditions} does not have an immediate cohomological version.

These are not the only possible definitions of boundary conditions. It is common to consider boundary conditions of the Potts model defined by specifying boundary spins directly, which leads to a more general notion than we consider here. Various subsets of these types of conditions are defined in~\cite{shklarov2023edwards}, where they are studied as subgroups of the full group of possible spin states using elementary group theory. As we are motivated by the limiting measure in $\Z^d,$ we will restrict ourselves to those which are consistent with some configuration of external plaquettes under the coupling, referred to as ``imprint boundary conditions'' in~\cite{shklarov2023edwards}. Since these arise from concrete cubical complexes, this allows us to take a more geometric approach. Notice that the special case of constant spins on the boundary of the domain does arise from an external plaquette configuration because it is equivalent to wired boundary conditions up to a choice of gauge (in topological terms, up to a choice of coboundary).

Next, we prove Lemma~\ref{lemma:UCTCboundaryfinite}.
\begin{proof}[Proof of Lemma~\ref{lemma:UCTCboundaryfinite}]
We have the following commutative diagram.

\[\begin{tikzcd}
H^{i-1}\paren{P_2;\;\Z_q} \arrow{r}{\phi^*} \arrow{d}{h_2} & H^{i-1}\paren{P_1;\;\Z_q} \arrow{d}{h_1} \\
\Hom\paren{H_{i-1}\paren{P_2;\;\Z_q},\Z_q} \arrow{r}{\circ \phi_*} & \Hom\paren{H_{i-1}\paren{P_1;\;\Z_q},\Z_q} \
\end{tikzcd}
\]

Here, $h_j:H^{i-1}\paren{P_j;\; \Z_q}\rightarrow \Hom\paren{H_{i-1}\paren{P_j;\;\Z_q},\Z_q}$ is the homomorphism induced by sending $\brac{f}\in  H^{i-1}\paren{P_j;\; \Z_q}$ to the homomorphism that sends $\brac{\sigma} \in H_{i-1}\paren{P_j;\;\Z_q}$ to $f\paren{\sigma}$ (this is well-defined by standard arguments; see Section 3.1. of~\cite{hatcher2002algebraic}). The maps $h_j$ are in fact isomorphisms by Proposition~\ref{prop:UCTC}. In addition, the lower horizontal row sends a homomorphism $f\in \Hom\paren{H_{i-1}\paren{P_2;\;\Z_q},\Z_q}$ to $f\circ \phi_* \in  \Hom\paren{H_{i-1}\paren{P_1;\;\Z_q},\Z_q}.$ By following the diagram, we have that 
\[\im\phi^* \cong \Hom\paren{\phi_*\paren{H_{i-1}\paren{P_1;\;\Z_q}},\Z_q}\]
and, in particular,
\[\abs{\im\phi^*}= \abs{\Hom\paren{\phi_*\paren{H_{i-1}\paren{P_1;\;\Z_q}},\Z_q}}=\abs{\im\phi_*}\,.\]
    
\end{proof}

We now show that the effects of an infinite volume boundary condition appear in sufficiently large finite approximations. We first prove a straightforward characterization of nullhomology in the Borel-Moore setting in terms of finite approximations. Although we work with Borel-Moore chains to streamline the proof, we remark that the second condition in the lemma is equivalent to an analogous one for ordinary homology.

\begin{Lemma}\label{lemma:BMapprox}
    Let $X \subset \Z^d$ be an $i$-dimensional percolation subcomplex, and let $\xi$ be boundary conditions. Then for any finitely supported $\gamma \in Z^{\mathrm{BM}}_{i-1}\paren{X;\;\Z_q},$ $0 = \brac{\gamma} \in H^{\mathrm{BM}}_{i-1}\paren{X;\;\Z_q}$ if and only if there is an $N \in \N$ so that for all $n>N,$ $\gamma$ is homologous to an $(i-1)$-cycle $\gamma_n$ supported on $\partial \Lambda_n.$ 
\end{Lemma}

\begin{proof}
    The forward implication is obvious, so suppose that for each $n>N$ there exists a $\tau_n\in C_i\paren{X \cap \Lambda_n;\;\Z_q}$ so that 
\begin{equation}
\label{eq:tau}
\partial \tau_n = \gamma +\gamma_n
\end{equation}
where $\gamma_n$ is supported on $\partial \Lambda_n.$ By a standard argument, we can choose the chains $\tau_n$ to be compatible in the sense that the restriction of $\tau_m$ to $\Lambda_n$ is $\tau_n$ for $m>n.$ Say that $\tau_{n_0}\in C^{\mathrm{BM}}_i\paren{X\cap \Lambda_{n_0};\;\Z_q}$ is \emph{extendable} if there exist infinitely many $n_1>n_0$ and  $\tau_{n_1}\in C^{\mathrm{BM}}_i\paren{X\cap \Lambda_{n_1};\;\Z_q}$ which satisfy (\ref{eq:tau}) and so that
$$\tau_{n_1}=\tau_{n_0}+\eta_{n_1,n_0}$$
for some $\eta_{n_1,n_0}$ supported outside of $\Lambda_{n_0}.$ 
As there are only finitely many choices of $\tau_{n_0}$ and the restriction of a chain satisfying that equation for a larger value of $n$ satisfies it for $n_0,$ there exists at least one extendable choice of $\tau_{n_0}.$ Suppose $n$ is large enough so that that $\gamma$ is supported on $\Lambda_{n}$ and choose an extendable chain $\tau_n.$ Since $\tau_n$ is extendable, there must exist an extendable choice of $\tau_{n+1}\in C^{\mathrm{BM}}_i\paren{X\cap \Lambda_{n+1};\;\Z_q}$ whose restriction to $\Lambda_n$ is $\tau_n.$ Continuing this construction one step at a time for all $m>n$ results in the desired compatible family of chains $\set{\tau_m}_{m\geq n}.$ Setting $\eta_{n}=\tau_{m+1}-\tau_m$ we obtain that
$$\eta\coloneqq \sum_{m=n}^{\infty} \eta_n$$
is an element of $C^{\mathrm{BM}}_i\paren{X;\;\Z_q}$ so that $\partial\eta=\gamma.$ 
\end{proof}

Given boundary conditions $\xi$ and a subcomplex $X \subset \Z^d,$ define 
\[\xi_{X} \coloneqq \xi \cap X\] 
and
\[\hat{\xi}_X \coloneqq \xi \cup \paren{\Z^d \setminus X}^{\paren{i}}\,.\]
Note that for a sufficiently large box $r,$ $\hat{\xi}_X$ and $\hat{\xi}_X \cap r$ have the same effect as boundary conditions for $X.$ We now use this observation to show that any boundary condition, including one of the type introduced in Definition~\ref{defn:wiredboundaryconditions}, can be replaced by a finite boundary condition from Definition~\ref{defn:boundaryconditions}. After that, we will see that any of these measures can be obtained from one of the form given in Definition~\ref{def:prcm} by conditioning.

\begin{Proposition}\label{prop:largeboundary}
    Let $r$ be a box in $\Z^d$ and let $\xi$ be boundary conditions for PRCM on $r.$ Then there is a cube $\Lambda_n$ containing $r$ so that for any  any $r' \supset \Lambda_n,$
    \[\mu_{r ,p,q,i}^{\xi} \stackrel{d}{=} \mu_{r ,p,q,i}^{\xi_{r'}}\,\]
    and 
    \[\mu_{r ,p,q,i}^{\overline{\xi}} \stackrel{d}{=} \mu_{r ,p,q,i}^{\hat{\xi}_{r'}}\,\]
\end{Proposition}

\begin{proof}
We first prove the statement for $\mu_{r ,p,q,i}^{\xi}.$ Fix a percolation subcomplex $P$ of $r.$ Here we will view the cluster weight term as homological rather than cohomological by applying Lemma~\ref{lemma:UCTCboundaryfinite} (so that we do not need to switch perspectives for the proof of the second statement). Roughly speaking, our goal is to show that  $\abs{\im\phi_*}$ is determined by some finite subcomplex of $P_{\xi},$ where 
\[\phi_*:H_{i-1}\paren{P;\;\Z_q}\to H_{i-1}\paren{P_{\xi};\;\Z_q}\]
is the homomorphism induced by the inclusion $P\hookrightarrow P_{\xi}.$ 

For each $\gamma\in \ker \phi_*,$ we can find a chain $\tau_{\gamma} \in C_i\paren{P_{\xi};\;\Z_q}$ so that $\partial \tau_{\gamma} = \gamma.$ By definition, $\tau_{\gamma}$ is supported on finitely many $i$-plaquettes, and therefore on $\Lambda_n \cap P_{\xi}$ for a sufficiently large $n.$ In fact, we may choose $n$ to be large enough so that for all $\gamma \in \ker \phi_*,$ $\tau_{\gamma}$ is supported on $\Lambda_n \cap P_{\xi}.$ Given $r' \supset \Lambda_n,$ define 
\[\phi^{r'}_*:H_{i-1}\paren{P;\; \Z_q}\rightarrow H_{i-1}\paren{r' \cap P_{\xi};\;\Z_q}\]
to be the map on inclusion. From the first isomorphism theorem we have that
\[\ker \phi_*=\ker \phi^{r'}_*\implies \abs{\im \phi_*}= \abs{\im \phi^{r'}_*}\,.\]
Since there are only finitely many percolation subcomplexes $P$ of $r,$ this equality holds for all of them when $n$ is sufficiently large. We can therefore find a $n$ so that 
\[\mu_{r ,p,q,i}^{\xi} \stackrel{d}{=} \mu_{r ,p,q,i}^{\xi_{r'}}\]
for all $r' \supset \Lambda_n.$

The case of $\mu_{r ,p,q,i}^{\overline{\xi}}$ is similar in spirit. Notice that the size of the kernel of the induced map on Borel--Moore homology 
$$\phi_*^n:H^{\mathrm{BM}}_{i-1}\paren{P;\;\Z_q}\to H^{\mathrm{BM}}_{i-1}\paren{P_{\hat{\xi}_{\Lambda_n} ;\;\Z_q}}$$
is decreasing in $n.$ As such, it suffices to show that there is an $n$ so that $\ker\phi_*^n\subseteq \ker \phi_*^{\mathrm{BM}}.$ It follows from Lemma~\ref{lemma:BMapprox} that for any fixed $P,$ $\gamma\in  \ker \phi_*^{\mathrm{BM}}$ if and only if $\gamma\in  \ker \phi_*^n$ for all sufficiently large $n.$ Then, as there are only finitely many choices of $P$ and $\gamma,$ we can choose $n$ large enough so that this is true for all of them.

\end{proof}

Although it is not important in the context of this paper, we remark that although there is not an obvious analogue of Lemma~\ref{lemma:UCTCboundaryfinite} for $\mu^{\overline{\xi}}_{r,p,q,i},$ one can use Proposition~\ref{prop:largeboundary} to give an alternative definition as a limit of measures defined via cohomology.

These results allow us to give an alternative topological proof of the fact that boundary conditions are compatible with conditioning on subcomplexes, which appears as Proposition 60 of~\cite{shklarov2023edwards}. We are also able to generalize slightly to the wired boundary conditions of Definition~\ref{defn:wiredboundaryconditions}. 

\begin{Corollary}\label{cor:boundarywelldefined}
    Let $r_1\subset r_2$ be boxes in $\Z^d$ and let $\xi_2$ be boundary conditions for $r_2.$ For a subset $\xi_1$ of the $i$-plaquettes of $r_2\setminus r_1$ let $A\paren{\xi_1}$ be the event that $P\cap r_2\setminus r_1=\xi_1.$ Then,
   if
    \[\mu_1=\mu_{r_1 ,p,q,i}^{\xi_1\cup \xi_2}\]
    is the random-cluster measure on $r_1$ with boundary conditions $\xi_1\cup \xi_2$ and 
   \[\mu_2=\paren{\mu_{r_2,p,q,i}^{\xi_2} \middle \| A\paren{\xi_1}}\Big|_{r_1}\]
    is the restriction to $r_1$ of the random-cluster with boundary conditions $\xi_2$ conditioned on the event $A\paren{\xi_1},$
   \[\mu_1 \stackrel{d}{=}\mu_2 \,.\]
   Likewise, setting
   \[\overline{\mu}_1=\mu_{r_1 ,p,q,i}^{\overline{\xi_1\cup \xi_2}}\]
   and
   \[\overline{\mu}_2=\paren{\mu_{r_2,p,q,i}^{\overline{\xi_2}} \middle \| A\paren{\xi_1}}\Big|_{r_1}\,,\]
   we have
   \[\overline{\mu}_1 \stackrel{d}{=} \overline{\mu}_2 \,.\]
\end{Corollary}

\begin{proof}
 We first prove the equality $\mu_1 \stackrel{d}{=} \mu_2.$ As a preliminary step, we reduce to the case where the second set of boundary conditions are free. By Proposition~\ref{prop:largeboundary} we can replace $\xi_2$ with boundary conditions $\xi_2'$ that do not contain any plaquettes outside of a larger box $r_3.$ Set $\mu_3=\mu_{r_3,p,q,i}^{\mathbf{f}}.$
Assuming the result for free boundary conditions, we have that 
\[\mu_1 \stackrel{d}{=}\paren{\mu_3 \Big|_{r_1} \middle \| A\paren{\xi_1 \cup \xi_2'}} \]
and
\[\mu_2\stackrel{d}{=}\paren{\paren{\mu_3 \middle \| A\paren{\xi_2'}}\Big|_{r_2} \middle \| A\paren{\xi_1}}\Big|_{r_1}\,,\]
which coincide.

Now, assume that $\xi_2=\emptyset.$ Let $P_1$ be a subcomplex of $r_1,$ let $P_2$ (respectively, $P_3$) be the percolation subcomplexes of $r_2$ (respectively, $\Z^d$) containing all open plaquettes of $P_1$ and $\xi_1.$ $P_2$ and $P_3$ thus have the same $i$-plaquettes, but different $(i-1)$-skeleta. For $1\leq j < k \leq 3$ let $\phi_{k,j}^*:H^{i-1}\paren{P_k;\;\Z_q}\rightarrow H^{i-1}\paren{P_j;\;\Z_q}$ be the map on cohomology induced by the inclusion $\phi^{k,j}:P_j\hookrightarrow P_k.$ $\phi_{3,2}^*$ is surjective and $\phi_{3,1}^*=\phi_{3,2}^* \circ \phi_{2,1}^*$ so
\[\abs{\im\phi_{3,1}^*}=\abs{\im \phi_{2,1}^*}\,.\]

It follows that that 
\[\mu_1\paren{P_1}\propto p^{\abs{P_1}}\paren{1-p}^{\abs{P_1^c}} \abs{\im \phi^{1,3}_*}=p^{\abs{P_1}}\paren{1-p}^{\abs{P_1^c}} \abs{\im \phi_{2,1}^*}\]
 and 
 \begin{align*}
  \mu_2\paren{P_1}=\mu_{r_2,p,q,i}^{\mathbf{f}}\paren{P_2}&\propto  p^{\abs{P_2}}\paren{1-p}^{\abs{P_2^c}}\abs{\im \phi_{3,2}^*}\\
  &\propto p^{\abs{P_1}}\paren{1-p}^{\abs{P_1^c}} \abs{H^{i-1}\paren{P_2;\;\Z_q}} \,
 \end{align*}
 where we have removed a factor that does not depend on $P_1.$

 It suffices to show that 
 \begin{equation}\label{eq:subcomplexcohomology}
     \abs{H^{i-1}\paren{P_2;\;\Z_q}}/ \abs{\im \phi^{2,1}_*}
 \end{equation}
 does not depend on the state of $P_1.$  Towards that end, we apply the long exact sequence of the pair $(P_2,P_1)$ (see page 199 of~\cite{hatcher2002algebraic}; more detail is given for the homological analogue on page 115): 
 \[\begin{tikzcd}
  0 \arrow{r}{} &  H^{i-1}\paren{P_2,P_1;\;\Z_q} \arrow{r}{\chi} & H^{i-1}\paren{P_2;\;\Z_q} \arrow{r}{\phi_{2,1}^*}  & H^{i-1}\paren{P_1;\;\Z_q}\,. 
 \end{tikzcd}
 \]
  The leftmost term --- corresponding to to $H^{i-2}\paren{P_1;\;\Z_q}$ --- vanishes because $P_1$ is a percolation subcomplexes. 
By exactness, $\chi$ is injective and 
 \[\ker\paren{\phi_{2,1}^*}=\im \chi \cong H^{i-1}\paren{P_2,P_1;\;\Z_q}\,.\]

 We claim that $H^{i-1}\paren{P_2,P_1;\;\Z_q}$ does not depend on the states of $i$-plaquettes of $P_1.$ Recall that $C^j\paren{P_2,P_1;\;\Z_q}$ is  is the group of $j$-cochains of $P_2$ that vanish on chains supported on $P_1,$ the relative coboundary map
\[\delta^{j}_{P_2,P_1}: C^{j}\paren{P_2,P_1;\;\Z_q} \to C^{j+1}\paren{P_2,P_1;\;\Z_q}\]
is the restriction of the usual coboundary map, and 
\[H^{i-1}\paren{P_2,P_1;\;\Z_q}=\ker \delta^{i-1}_{P_2,P_1}/\im \delta^{i-2}_{P_2,P_1}\,.\]
The $(i-1)$-skeletons of $P_2$ and $P_1$ do not depend on the state of their $i$-plaquettes, so neither do $C^{i-1}\paren{P_2,P_1}$ nor $C^{i-2}\paren{P_2,P_1}.$ It follows that $\im \delta^{i-2}_{P_2,P_1}$ and $\ker \delta^{i-1}_{P_2,P_1}$ are also independent of $P_1$ (changing the codomain of a map does not change its kernel), and thus $H^{i-1}\paren{P_2,P_1;\;\Z_q}$ is as well, completing the proof of the first statement.  

We use Proposition~\ref{prop:largeboundary} as a shortcut to prove that $\overline{\mu}_1 \stackrel{d}{=} \overline{\mu}_2.$ Choose a box $r_3\supset r_2$ large enough so that  
$$\mu_{r_2 ,p,q,i}^{\overline{\xi}_2} \stackrel{d}{=} \mu_{r_2 ,p,q,i}^{\hat{\xi_2}_{r_3}}$$
and
$$\mu_{r_1 ,p,q,i}^{\overline{\xi}_1} \stackrel{d}{=} \mu_{r_1 ,p,q,i}^{\hat{\xi_1\cup\xi_2}_{r_3}}.$$
Since $\hat{\xi_1\cup\xi_2}_{r_3}=\xi_1\cup \hat{\xi_2},$ the desired statement follows from the one from non-wired boundary conditions. 
\end{proof}

\section{Duality and Boundary Conditions}
We now consider boundary conditions in dual complexes. Recall that complex that we have defined on $\Z^d$ has an associated dual complex $\paren{\Z^d}^{\bullet} \coloneqq \Z^d + \paren{1/2,1/2,\ldots,1/2}.$ Each $i$-cell of $\Z^d$ intersects exactly one $(d-i)$-cell of $\paren{\Z^d}^{\bullet},$ so an $i$-dimensional percolation complex comes with a complementary dual $(d-i)$-dimensional dual complex.
Let $r$ be a box, let $P$ be an $i$-dimensional percolation subcomplex of $r,$ and let $Q$ be the dual complex. Recall our convention that $r$ does not contain any boundary $i$-plaquettes and includes the entire $(i-1)$-skeleton. Also, $Q$ is a subcomplex of $\overline{r^{\bullet}},$ by which we mean that it is allowed to contain $(d-i)$-plaquettes in the boundary. For convenience, set $\overline{Q}=Q\cup \partial r^{\bullet}.$ We will explore the relationship between $H^{i-1}\paren{P;\;\Z_q}$ and $H^{d-i-1}\paren{\overline{Q};\;\Z_q}$ by expressing them both in terms of $H_{i-1}\paren{P;\;\Z}.$ 

We begin by recalling a few standard topological tools. Homology and cohomology groups with different coefficient groups are related by what is called the \emph{Universal Coefficient Theorem}, which has a version for homology and cohomology (Theorems 3A.3 and 3.2 in \cite{hatcher2002algebraic} respectively). The following is a consequence of the Universal Coefficient Theorem for Homology:
\begin{equation}
\label{eq:UCTH}
H_j\paren{P;G}\cong \paren{H_j\paren{P;\;\Z}\otimes G}\oplus \mathrm{Tor}\paren{H_{j-1}\paren{P;\;\Z},G}\,.
\end{equation}

Then if we write $H_j\paren{X;\;\Z}=\Z^{\mathbf{b}_j\paren{X}}\oplus T_j\paren{X},$ the Universal Coefficient Theorem for Cohomology yields that
\begin{equation}
    \label{eq:UCTCZ}
  H^j\paren{X;\;\Z}\cong \paren{\paren{H_j\paren{X;\;\Z},\Z}/T_j\paren{X}}\oplus T_{j-1}\paren{X}\,
\end{equation}
(this is stated as Corollary 3.3 in~\cite{hatcher2002algebraic}).

We also use a formulation of Alexander duality in percolation complexes previously given in~\cite{duncan2023sharp}.

\begin{Proposition}
\label{cor:alexander}
Fix $0<i<d$ and a box $r$ in $\Z^d.$ If $P$ is a percolation subcomplex of $\overline{r}$ ($r$), $Q$ is the dual complex, and $r'$ is the box $r^{\bullet}$ (respectively $\overline{r^{\bullet}}$) then there is an isomorphism
 \[\mathcal{I} : H_{i}\paren{P_r;\;\Z} \to H^{d-i-1}\paren{Q\cup \partial r';\; \Z}\]
 where $H_j\paren{X;\;\Z}$ and $H^j\paren{X;\;\Z}$ denote the $j$-dimensional reduced homology and the $j$-dimensional reduced cohomology of $X$ with integral coefficients.
\end{Proposition}

Combining these facts yields the following proposition.

\begin{Proposition}\label{prop:torsion}
Let $q\in \N+1.$ Then
\[H^{i-1}\paren{P;\;\Z_q}\cong \Z_q^{\mathbf{b}_{i-1}\paren{P;\;\Z}} \oplus \mathrm{Tor}\paren{H_{i-1}\paren{P;\;\Z}, \Z_q}\,.\]
and 
\[\tilde{H}^{d-i-1}\paren{\overline{Q};\;\Z_q}\cong \Z_q^{\mathbf{b}_{i}\paren{P;\;\Z}} \oplus \mathrm{Tor}\paren{H_{i-1}\paren{P;\;\Z}, \Z_q}\,.\]
In particular, there is a constant $c=c\paren{N,i,d}$  so that
\begin{equation}
\label{eq:EP}
\abs{H^{i}\paren{P;\;\Z_q}}=\abs{H^{d-i-1}\paren{\overline{Q};\;\Z_q}}q^{c-\abs{P}}\,.
\end{equation}
\end{Proposition}
\begin{proof}
For the first claim,
\begin{align*}
    H^{i-1}\paren{P;\;\Z_q}&\cong  H_{i-1}\paren{P;\;\Z_q} &&\text{Proposition~\ref{prop:UCTC}}\\
    &\cong H_{i-1}\paren{P;\;\Z}\otimes \Z_q && \text{(\ref{eq:UCTH}), $H_{i-2}\paren{P;\;\Z}\cong 0$}\\
    &\cong \Z_q^{\mathbf{b}_{i-1}\paren{P;\;\Z}} \oplus \mathrm{Tor}\paren{H_{i-1}\paren{P;\;\Z}, \Z_q}  && \text{properties of $\otimes$}\,.
\end{align*}

We now demonstrate the second claim.
\begin{align*}
    H^{d-i-1}\paren{\overline{Q};\;\Z_q}&\cong  H_{d-i-1}\paren{\overline{Q};\;\Z_q} &&\text{Proposition~\ref{prop:UCTC}}\\
    &\cong \paren{H_{d-i-1}\paren{\overline{Q};\;\Z}\otimes \Z_q} && \text{(\ref{eq:UCTH}), $H_{d-i-2}\paren{\overline{Q};\;\Z}\cong 0$}\\
    &\cong \paren{H^{i}\paren{P;\;\Z}\otimes \Z_q} &&\text{Corollary~\ref{cor:alexander}}\\
    &\cong \Z_q^{\mathbf{b}_{i}\paren{P;\;\Z}} \oplus \mathrm{Tor}\paren{H_{i-1}\paren{P;\;\Z}, \Z_q} &&\text{(\ref{eq:UCTCZ}), properties of $\otimes$}\,.
\end{align*}

Finally,  $H_0\paren{P;\;\Z}\cong \Z,$ $H_i\paren{P;\;\Z},$ and  $H_{i-1}\paren{P;\;\Z}$ are the only non-zero homology groups of $P,$ so the Euler--Poincar\'{e} theorem (Theorem 2.44 in~\cite{hatcher2002algebraic}) yields that
\[\chi\paren{P}=1+(-1)^{i-1} \mathbf{b}_{i-1}\paren{P;\;\Z}+(-1)^i\mathbf{b}_{i}\paren{P;\;\Z}=\abs{P}+\sum_{j=0}^{i-1}\abs{P^{(j)}}\,.\]
Then (\ref{eq:EP}) follows because the number of $j$-dimensional plaquettes in $P$ ($\abs{P^{(j)}}$) does not depend on $P$ for $j\leq i-1.$ 
\end{proof}

The following special case of Alexander duality is reproduced from~\cite{duncan2023sharp}.

\begin{Proposition}
\label{prop:alexander}
Fix $0<i<d$ and a box $r$ in $\Z^d.$ If $P$ is a percolation subcomplex of $\overline{r}$ ($r$), $Q$ is the dual complex, and $r'$ is the box $r^{\bullet}$ (respectively $\overline{r^{\bullet}}$) then there is an isomorphism
 \[\mathcal{I} : H_{i}\paren{P_r;\;\Z} \to H^{d-i-1}\paren{Q\cup \partial r';\; \Z}\]
 where $H_j\paren{X;\;\Z}$ and $H^j\paren{X;\;\Z}$ denote the $j$-dimensional reduced homology and the $j$-dimensional reduced cohomology of $X$ with integral coefficients.
\end{Proposition}

 We are now ready to prove that the $i$-dimensional PRCM with free boundary conditions on $r$ is dual to a $(d-i)$-dimensional wired PRCM on $\overline{r^{\bullet}}.$ The argument is similar to the one used in~\cite{duncan2022topological}, with the key difference being the use of the preceding proposition.
 
\begin{Theorem}\label{theorem:dualitygeneral}
Let $q\in \N+1$ and $1\leq i \leq d-1.$ Also, define
\begin{equation}
    \label{eq:pstar}
    p^*=p^*(p,q)= \frac{\paren{1-p}q}{\paren{1-p}q + p}\,.
\end{equation}
Then, if $r$ is a box in $\Z^d,$
\begin{equation*}
   \mu^{\mathbf{f}}_{r,p,q,i}\paren{P} = \mu^{\mathbf{w}}_{\overline{r^\bullet},p^*,q,d-i}\paren{Q}\
\end{equation*}
and
\begin{equation*}
   \mu^{\mathbf{f}}_{\overline{r},p,q,i}\paren{P} = \mu^{\mathbf{w}}_{r^\bullet,p^*,q,d-i}\paren{Q}\
\end{equation*}
\end{Theorem}

\begin{proof}
In order to show the first claim, we compute
\begin{align*}
  \mu^{\mathbf{f}}_{r,p,q,i}\paren{P} &=\frac{1}{\tilde{Z}}p^{\abs{P}}\paren{1-p}^{\abs{r^{(i)}}-\abs{P}}\abs{H^{i-1}\paren{P;\;\Z_q}}\\
    &= \frac{\paren{1-p}^{\abs{r^{(i)}}}}{\tilde{Z}}\paren{\frac{p}{1-p}}^{\abs{P}}\abs{H^{i-1}\paren{P;\;\Z_q}}\\
    &= \frac{\paren{1-p}^{\abs{r^{(i)}}}}{\tilde{Z}}\paren{\frac{p}{1-p}}^{\abs{P}}\abs{H^{d-i-1}\paren{\overline{Q};\;\Z_q}}q^{c-\abs{P}} &&\text{(\ref{eq:EP})}\\
    &= \frac{q^{c}\paren{1-p}^{\abs{r^{(i)}}}}{\tilde{Z}}\paren{\frac{q(1-p)}{p}}^{-\abs{P}}\abs{H^{d-i-1}\paren{\overline{Q};\;\Z_q}}\\
    &= \frac{q^{c}\paren{1-p}^{\abs{r^{(i)}}}}{\tilde{Z}}\paren{\frac{q(1-p)}{p}}^{\abs{\overline{Q}}-\abs{r^{(i)}}}\abs{H^{d-i-1}\paren{\overline{Q};\;\Z_q}}\\
    &= \frac{q^{c}\paren{1-p}^{\abs{r^{(i)}}}}{\tilde{Z}}\paren{\frac{p^*}{1-p^*}}^{\abs{\overline{Q}}-\abs{r^{(i)}}}\abs{H^{d-i-1}\paren{\overline{Q};\;\Z_q}} &&\text{(\ref{eq:pstar})}\\
    &= \frac{q^{c}\paren{1-p}^{\abs{r^{(i)}}}}{\paren{p^*}^{\abs{r^{(i)}}}\tilde{Z}}\frac{\paren{p^*}^{\abs{\overline{Q}}}}{\paren{1-p^*}^{\abs{\overline{Q}} - \abs{r^{(d-i)}}}}\abs{H^{d-i-1}\paren{\overline{Q};\;\Z_q}}\\
    &\coloneqq \frac{1}{\tilde{Z}^\bullet}\paren{p^*}^{\abs{\overline{Q}}}\paren{1-p^*}^{\abs{r^{(d-i)}}- \abs{Q}}\abs{H^{d-i-1}\paren{\overline{Q};\;\Z_q}}\\
    &= \mu^{\mathbf{w}}_{r^\bullet,p^*,q,d-i}\paren{Q}\,.
\end{align*}
The proof of the second claim is nearly identical, but uses the parenthetical formulation of Proposition~\ref{prop:alexander}.
\end{proof}

Since our notion of boundary conditions corresponds to a fixed external percolation complex, a natural set of dual boundary conditions is given by the dual complex. More precisely, for a box $r$ in $\Z^d$ and boundary conditions $\xi,$ the dual measure on $\overline{r^{\bullet}}$ is defined by setting $\xi^{\bullet}$ to include all $(d-i)$-plaquettes dual to closed $i$-plaquettes of $\xi.$ All of the pieces are now in place to describe the distribution of the dual complex to the PRCM with boundary conditions.

\begin{Theorem}\label{thm:mostgeneralduality}
   \[\mu^{\xi}_{r,p,G,i}\paren{P} = \mu^{\overline{\xi^\bullet}}_{\overline{r^{\bullet}},p^*,G,d-i}\paren{Q}\]
\end{Theorem}

\begin{proof}
By Proposition~\ref{prop:largeboundary}, we can choose a box $r_2$ large enough so that so that 
\[\mu_{r ,p,G,i}^{\xi} \stackrel{d}{=} \paren{\mu_{r_2,p,G,i}^{\mathbf{f}} \middle \| A\paren{\xi}}\Big|_{r}\]
and 
\[\mu_{\overline{r^{\bullet}} ,p^*,G,d-i}^{\overline{\xi^{\bullet}}} \stackrel{d}{=} \paren{\mu_{\overline{r_2^{\bullet}},p^*,G,d-i}^{\mathbf{w}} \middle \| A\paren{\xi^{\bullet}}}\Big|_{\overline{r^{\bullet}}}\,.\]
We can conclude by applying Theorem~\ref{theorem:dualitygeneral}. 
\end{proof}

As a corollary, we have that $\abs{\im\phi^*}=\abs{\im \psi_*^{\mathrm{BM}}}.$ This identity can be proven directly using a commutative diagram involving long exact sequences of the pairs $(P_\xi,P)$ and $(\overline{Q}\cup \overline{Q_{\xi^{\bullet}}},\overline{Q})$ together with Poincar\'{e} and Lefschetz dualities. Then one could prove Proposition~\ref{prop:largeboundary} using the same argument as in Theorem~\ref{theorem:dualitygeneral}. However, the resulting proof would be longer.

\section{The Infinite Volume Limit}

In this section, we apply our previous results to understand the random complexes on $\Z^d$ constructed as weak limits of finite volume PRCMs with boundary conditions. First, we recall some basic tools to compare different measures. For two measures $\mu_1$ and $\mu_2$ on percolation subcomplexes, we say that $\mu_1$ is \emph{stochastically dominated} by $\mu_2$ and write $\mu_1 \leqst \mu_2$ if there is a coupling $\kappa$ of random complexes $P_1$ and $P_2$ distributed according to $\mu_1$ and $\mu_2$ respectively so that 
\[\kappa\paren{P_1 \subset P_2} = 1.\]

The main tool we will use to show stochastic domination is due to Holley~\cite{holley1974remarks}.

\begin{Theorem}[Holley's Inequality]\label{thm:holley}
Let $I$ be a finite index set and let 
\[\paren{X_i}_{i \in I},\paren{Y_i}_{i \in I} \in \set{0,1}^I\]
be random vectors distributed according to strictly positive probability measures $\mu_1$ and $\mu_2.$ Suppose that for each pair $\paren{W_i}_{i \in I},\paren{Z_i}_{i \in I} \in \set{0,1}^I$ with $W_i \leq Z_i$ for each $j \in I,$ 
\begin{align*}
    &\mu_1\paren{X_j = 1 : X_i = W_i \text{ for all } i \in I \setminus \set{j}}\\ 
    &\qquad\leq \mu_2\paren{Y_j = 1 :  Y_i = Z_i \text{ for all } i \in I \setminus \set{j}}\,.
\end{align*}
Then $\mu_1 \leqst \mu_2.$
\end{Theorem}

We can now compare measures in subcomplexes in the free and wired cases.

\begin{Lemma}\label{lemma:freewiredmono}
    Let $X \subset Y$ be subcomplexes of $\Z^d.$ Then $\mu^{\mathbf{f}}_{X,p,G,i} \leqst \paren{\mu^{\mathbf{f}}_{Y,p,G,i}}\big|_{X}$
    and   $\paren{\mu^{\mathbf{w}}_{Y,p,G,i}}\big|_X \leqst \mu^{\mathbf{w}}_{X,p,G,i}.$
\end{Lemma}

\begin{proof}
    This is a straightforward consequence of Theorem~\ref{thm:holley}.
\end{proof}

This gives a quick proof that the free and wired limits exist.
\begin{Proposition}\label{prop:freewired}
    The limits 
    \[\mu_{\Z^d,p,q,i}^{\mathbf{f}} \coloneqq \lim_{n \to \infty} \mu^{\mathbf{f}}_{\Lambda_n,p,q,i}\,\] 
    and 
    \[\mu_{\Z^d,p,q,i}^{\mathbf{w}} \coloneqq \lim_{n \to \infty} \mu^{\mathbf{w}}_{\Lambda_n,p,q,i}\,\]
    exist.
\end{Proposition}

\begin{proof}
    By Lemma~\ref{lemma:freewiredmono}, for fixed $m,$ $\paren{\mu^{\mathbf{f}}_{\Lambda_n,p,q,i}}\big|_{\Lambda_m}$ and $\paren{\mu^{\mathbf{w}}_{\Lambda_n,p,q,i}}\big|_{\Lambda_m}$ are increasing and decreasing in $n$ respectively, so they must converge.
\end{proof}

As in the classical RCM, it is not hard to see that the free and wired boundary conditions are extremal.

\begin{Proposition}\label{prop:extremalmeasures}
Let $r$ be a box in $\Z^d$ and let $\xi$ be any boundary conditions. Then    
\[\mu^{\mathbf{f}}_{r,p,G,i} \leqst\mu^{\xi}_{r,p,G,i}  \leqst\mu^{\mathbf{w}}_{r,p,G,i}\,.\]

In addition, for any $\mu_{\Z^d,p,G,i}$ which is an infinite volume random-cluster measure which is a weak limit of measures on boxes with boundary conditions, 

\begin{equation}\label{eq:extremalmeasures}
    \mu^{\mathbf{f}}_{\Z^d,p,G,i} \leqst\mu_{\Z^d,p,G,i}  \leqst\mu^{\mathbf{w}}_{\Z^d,p,G,i}\,.  
\end{equation}

\end{Proposition}

As a consequence of Theorem~\ref{thm:mostgeneralduality}, there is also a relationship between dual limiting measures, when they exist.

\begin{Theorem}\label{thm:limduality}
    Suppose the weak limit
    \[\mu^\xi_{\Z^d,p,q,i} \coloneqq \lim_{n \to \infty} \mu^{\xi}_{\Lambda_n,p,q,i}\]
    exists. Then the weak limit
    \[\mu^{\overline{\xi^{\bullet}}}_{\paren{\Z^d}^{\bullet},p^*,q,d-i} \coloneqq \lim_{n \to \infty} \mu^{\overline{\xi^{\bullet}}}_{\overline{\Lambda_n^{\bullet}},p^*,q,d-i}\]
    also exists, and satisfies
    \[\mu^\xi_{\Z^d,p,q,i}\paren{P}  \stackrel{d}{=} \mu^{\overline{\xi^{\bullet}}}_{\paren{\Z^d}^{\bullet},p^*,q,d-i}\paren{Q}\,.\]
\end{Theorem}

We also see that the general free and wired measures are dual.

\begin{Corollary}
    The free and wired measures $\mu^{\mathbf{f}}_{\Z^d,p,q,i}$ and $\mu^{\mathbf{w}}_{\Z^d,p,q,i}$ satisfy
    \[\mu^{\mathbf{f}}_{\Z^d,p,q,i} \paren{P} \stackrel{d}{=} \mu^{\mathbf{w}}_{\Z^d,p^*,q,d-i}\paren{Q}\]
\end{Corollary}

\begin{proof}
    The statement essentially follows from Theorem~\ref{thm:limduality}. Since the choice between Definitions~\ref{defn:boundaryconditions} and~\ref{defn:wiredboundaryconditions} is irrelevant for wired boundary conditions, we need only check that 
    \[\lim_{n\to\infty} \mu^{\mathbf{w}}_{\overline{\Lambda_n},p,q,i} \stackrel{d}{=} \lim_{n\to\infty} \mu^{\mathbf{w}}_{\Lambda_n,p,q,i}\,.\]
    This follows from Lemma~\ref{lemma:freewiredmono}, since for all $n,$ we have
    \[\Lambda_n \subset \overline{\Lambda_n} \subset \Lambda_{n+1}\,.\]
\end{proof}

\section{Uniqueness of the Infinite Volume Measure}

In the classical RCM with a fixed parameter $q,$ the wired and free infinite volume measures are known to coincide except possibly at a countable set of values of $p.$ In this section we adapt Grimmett's proof of this result~\cite{grimmett1995stochastic} to the PRCM; only minor modifications are required. One application of this result comes from Proposition 34 in~\cite{duncan2023sharp}: if $p$ is such that there is an infinite volume PRCM, two notions of surface tension given in terms of the asymptotic probability that an $(i-1)$-cycle $\gamma$ is null-homologous coincide. When $q\in \N + 2,$ this in turn implies that two definitions of Wilson loop tension agree in the coupled PLGT.  

Fix $1 \leq i \leq d,$ let $r$ be a box, let $\xi$ be boundary conditions and let $\Omega_r$ be the set of $i$-dimensional percolation subcomplexes on $r.$ We work with a slight modification of the partition function for the PRCM on $r$ with boundary conditions $\xi,$ namely
\begin{align*}
    Y^{\xi}_{r,p,q} &\coloneqq \paren{1-p}^{-\abs{r^{\paren{i}}}}Z^{\xi}_{r,p,q} = \paren{1-p}^{-\abs{r^{\paren{i}}}}\sum_{P \in \Omega_r} p^{\abs{P}}\paren{1-p}^{\abs{r^{\paren{i}}}-\abs{P}}\abs{\im \phi^*}\\
    &= \sum_{P \in \Omega_r} \abs{\im \phi^*} \exp\paren{\pi\abs{P}}\,,
\end{align*}
where $Z^{\xi}_{r,p,q}$ is the usual partition function and $\pi \coloneqq \log\paren{p/\paren{1-p}}.$ We now consider a notion of pressure. For a box $r$ and boundary conditions $\xi$, write 
\[f^{\xi}_r\paren{p,q} \coloneqq \frac{1}{\abs{r^{\paren{i}}}} \log Y^{\xi}_{r,p,q}\,.\]

\begin{Proposition}\label{prop:pressure}
    Let $\set{r_k}$ be an increasing sequence of boxes with $\bigcup_k r_k = \Z^d$ and let $\xi$ be a set of boundary conditions. Then the limit
    \[f\paren{p,q} = f^{\xi}\paren{p,q} \coloneqq \lim_{k \to \infty} f^{\xi}_{r_k}\paren{p,q}\]
    exists and does not depend on the choice of $\set{r_k}$ or $\xi.$ Moreover, $f\paren{p,q}$ is a convex function of $\pi$ and therefore differentiable as a function of $p \in \paren{0,1}$ except possibly on a countable set.
\end{Proposition}

\begin{proof}
We first verify that the limit $f^{\mathbf{f}}\paren{p,q}$ exists and does not depend on $\set{r_k}$. For simplicity of presentation, we will show this in the case of cubes and leave the extension to general boxes as an exercise. Note that if 
$\Lambda_m=\brac{-m,m}^d,$ we have for each $0 \leq j \leq d$ that 
\[\lim_{m \to \infty} \frac{\abs{\Lambda_m^{\paren{j}}}}{m^d}=2^d\binom{d}{j}\]
so 
 $\lim_{m \to \infty} f^{\mathbf{f}}_{\Lambda_m}\paren{p,q}$ exists if and only if  
 \[\lim_{m \to \infty} \frac{1}{m^d} Y^{\mathbf{f}}_{\Lambda_m,p,q} = -\log\paren{1-p} \lim_{m \to \infty} \frac{1}{m^d} Z^{\mathbf{f}}_{\Lambda_m,p,q}\] does.

We now recall a standard tool for understanding the topology of a space from that of its subspaces. The \emph{Mayer--Vietoris Sequence for Cohomology} relates the cohomology of a union of two spaces to the cohomology groups of the spaces and their intersection. See Sections 2.2 and 3.2 of~\cite{hatcher2002algebraic} for details. It is an exact sequence, meaning that the image of one  map in the sequence is the kernel of the next. Suppose $X = A \cup B.$ Then we apply the Mayer--Vietoris sequence to the decomposition $P_X = P_A \cup P_B$ where we denote $P\cap Y$ by $P_Y.$  
    \[\begin{tikzcd}[cramped, sep = small]
    H^{i-1}\paren{P_X} \arrow{r}{\varphi}  & H^{i-1}\paren{P_A} \oplus H^{i-1}\paren{P_B} \arrow{r}{\chi} & H^{i-1}\paren{P_{A\cap B}} 
    \end{tikzcd}
    \,.\]
    Now by the first isomorphism theorem,
    \begin{equation*}
        \frac{\abs{H^{i-1}\paren{P_A;\;\Z_q}}\abs{H^{i-1}\paren{P_B;\;\Z_q}}}{\abs{H^{i-1}\paren{P_X;\;\Z_q}}}\leq \abs{H^{i-1}\paren{P_{A\cap B};\;\Z_q}}
    \end{equation*}
    with equality holding if $A$ and $B$ are disjoint. We also have for any subcomplex $Y \subset X,$
    \[\abs{H^{i-1}\paren{Y;\;\Z_q}} \leq q^{\abs{Y^{\paren{i-1}}}}\,.\]
    Now, assume that the $i$-skeleta of $A$ and $B$ are disjoint. Combining the previous inequalities and summing over the $i$-plaquettes of $A\cup B$ yields
 \begin{equation}\label{eq:MV}
     \frac{1}{q^{\abs{A\cap B^{\paren{i-1}}}}}  Z^{\mathbf{f}}_{A,p,q}Z^{\mathbf{f}}_{B,p,q} \leq Z^{\mathbf{f}}_{A \cup B,p,q} 
 \end{equation}

    Now let $\Lambda_n \subset \Lambda_m.$ Consider a maximal packing of $\Lambda_m$ by disjoint translates $\set{\Lambda^l}_{l=1}^{k}$ of $\Lambda_n$ where 
    $$\paren{\frac{m}{n+1}-1}^d\leq k=k\paren{m,n}$$    
    Denote the union of these translates by $A$ and note that the disjointness of the cubes implies that 
    \begin{equation}\label{eq:disjoint}
            Z^{\mathbf{f}}_{A,p,q}=\paren{Z^{\mathbf{f}}_{\Lambda_n,p,q}}^k\,.
    \end{equation}

 Let $B$ be the induced subcomplex of $\Lambda_n$ containing all $i$-plaquettes that are not in any of the $\Lambda^l.$ $B$ is contained in the union of $m^d-n^d k $ unit $d$-cubes 
 so for any subcomplex $B_0$ of $B$ and any $1\leq j\leq d$ 
      \[\abs{B_0^{\paren{j}}}\leq c_j\paren{m^d-n^d k}\,,\]
      where $c_j \coloneqq 2^{2d-j}\binom{d}{j}.$ As a consequence,
      $$\log\paren{Z^{\mathbf{f}}_{B,p,q}}\leq \abs{B^{\paren{i-1}}}\leq c_{i-1}\log\paren{q} \paren{m^d-n^d k}\,.$$

Then, by Equations \ref{eq:MV} and \ref{eq:disjoint}, for fixed $n$ 
\begin{align*}
&\frac{1}{m^d}\log\paren{Z^{\mathbf{f}}_{\Lambda_m,p,q}}\\
&\qquad\geq  \frac{1}{m^d}\paren{ k \log\paren{Z^{\mathbf{f}}_{\Lambda_n,p,q}}+\log\paren{Z^{\mathbf{f}}_{B,p,q}} - \abs{A\cap B^{\paren{i-1}}}\log\paren{q}}\\
&\qquad \geq  \paren{\frac{1}{n+1}-\frac{1}{m}}^d\log\paren{Z^{\mathbf{f}}_{\Lambda_n,p,q}} -c_{i-1}\log\paren{q} \paren{1-\paren{\frac{n}{n+1}-\frac{n}{m}}^d}\,.
\end{align*}

So
\begin{align*}
&\liminf_{m\to\infty} \frac{1}{m^d}\log\paren{Z^{\mathbf{f}}_{\Lambda_m,p,q}}\\
&\qquad 
\begin{aligned}
\geq\limsup_{n\to\infty} \liminf_{m\to\infty} 
&\left[\paren{\frac{1}{n+1}-\frac{1}{m}}^d \log\paren{Z^{\mathbf{f}}_{\Lambda_n,p,q}}\right.\\
&\quad\left.-c_{i-1}\log\paren{q} \paren{1-\paren{\frac{n}{n+1}-\frac{n}{m}}^d}\right]
\end{aligned}\\
&\qquad\begin{aligned}\geq \limsup_{n\to\infty} \brac{\frac{1}{\paren{n+1}^d} \log\paren{Z^{\mathbf{f}}_{\Lambda_n,p,q}}-c_{i-1}\log\paren{q} \paren{1-\paren{\frac{n}{n+1}}^d}}\end{aligned}\\
&\begin{aligned}\qquad=\limsup_{n\to\infty} \frac{1}{n^d}\log\paren{Z^{\mathbf{f}}_{\Lambda_n,p,q}}\end{aligned}\,,
\end{align*}
and
$f^{\mathbf{f}}\paren{p,q}$ exists.
    
    Since $\abs{H^{i-1}\paren{\partial r_k;\;\Z_q}} = q^{\abs{\paren{\partial r_k}^{\paren{i}}}-1},$ we have
    \[Y^{\mathbf{f}}_{r_k,p,q} \leq Y^{\xi}_{r_k,p,q} \leq Y^{\mathbf{w}}_{r_k,p,q} \leq Y^{\mathbf{f}}_{r_k,p,q}q^{\abs{\paren{\partial r_k}^{\paren{i}}}-1}\,.\]
    Notice that
    \[\frac{\log\paren{q^{\abs{\paren{\partial r_k}^{\paren{i}}}}}}{\abs{r_k^{\paren{i}}}} \xrightarrow[]{k \to \infty} 0\]
    so the existence of the limit $f^{\mathbf{f}}\paren{p,q}$ implies that $f\paren{p,q}$ exists and is well defined.
    As in the classical RCM, taking the second derivative of $f^{\mathbf{f}}_{r_k}$ shows that it is a convex function of $\pi,$ so the limiting function $f\paren{p,q}$ is also convex. Since $\pi$ is a differentiable function of $p,$ we therefore have that $f\paren{p,q}$ is also differentiable outside of a countable set.
\end{proof}

\begin{Theorem}\label{thm:uniqueness}
    Suppose $f\paren{p,q}$ is differentiable as a function of $p$ at $p = p_0.$ Then there is a unique infinite volume PRCM with parameters $p_0,q.$
\end{Theorem}

\begin{proof}
    Let $\xi \in \set{\mathbf{f},\mathbf{w}}$ and let $h^{\xi}\paren{p,q} = \mu^{\xi}_{\Z^d,p,q,i}\paren{\text{$\sigma$ is open}}.$ Since both measures are easily checked to be translation invariant, this value does not depend on the choice of plaquette $\sigma.$ We now compare $h^{\mathbf{f}}\paren{p,q}$ and $h^{\mathbf{w}}\paren{p,q}.$ Notice that for any box $r,$
    \begin{equation}\label{eq:dfdpi}
        \frac{d f^{\xi}_r\paren{p,q}}{d \pi} = \frac{1}{\abs{r^{\paren{i}}}} \mathbb{E}_{\mu^{\xi}}\paren{\abs{\set{\sigma \in r^{\paren{i}} \text { open}}}}\,.
    \end{equation}
    Then for any $\sigma_0 \in r^{\paren{i}},$ it follows from translation invariance and Proposition~\ref{eq:extremalmeasures} that
    \begin{align*}
         \frac{1}{\abs{r^{\paren{i}}}} \mathbb{E}_{\mu^{\mathbf{f}}}\paren{\abs{\set{\sigma \in r^{\paren{i}} \text { open}}}} & \leq \mu^{\mathbf{f}}_{\Z^d,p,q,i}\paren{\sigma_0 \text{ is open}} \\ 
         &\leq \mu^{\mathbf{w}}_{\Z^d,p,q,i}\paren{\sigma_0 \text{ is open}}\\
         & \leq \frac{1}{\abs{r^{\paren{i}}}} \mathbb{E}_{\mu^{\mathbf{w}}}\paren{\abs{\set{\sigma \in r^{\paren{i}} \text { open}}}}\,.
    \end{align*}
   
    Write $\pi_0 = \log\paren{p_0/\paren{1-p_0}}.$ Now from Proposition~\ref{prop:pressure} and convexity,
    \[\lim_{k \to \infty} \frac{d f^{\xi}_{r_k}\paren{p,q}}{d \pi} \paren{\pi_0} = \frac{d f\paren{p,q}}{d \pi} \paren{\pi_0}\,,\]
    so it follows that $h^{\mathbf{f}}\paren{p_0,q} = h^{\mathbf{w}}\paren{p_0,q}.$ Since $\mu^{\mathbf{f}}_{\Z^d,p_0,q,i} \leqst \mu^{\mathbf{w}}_{\Z^d,p_0,q,i}$ by Proposition~\ref{prop:extremalmeasures}, a standard comparison of cylinder events (Proposition 4.6 in~\cite{grimmett2006random}) gives 
    \[\mu^{\mathbf{f}}_{\Z^d,p_0,q,i} \,{\buildrel d \over =}\, \mu^{\mathbf{w}}_{\Z^d,p_0,q,i}\,,\]
    and all measures at $p_0$ coincide.
\end{proof}

\section{Positive Association in the PRCM}\label{sec:RCMproperties}

In this section we give an alternative proof using algebraic topology of the result of~\cite{shklarov2023edwards} that the PRCM is positively associated. First we adapt the proof of~\cite{hiraoka2016tutte} for the PRCM with coefficients in a field. 

\begin{Theorem}\label{thm:FKGgeneral}
Let $X$ be a finite cubical complex. Then $\mu_{X,p,q,i}$ is positively associated in the sense that for any two increasing events $A$ and $B$ 
\begin{equation*}
     \mu_{X,p,q,i}\paren{A\cap B} \geq \mu_{X,p,q,i}\paren{A}\mu_{X,p,q,i}\paren{B}.
 \end{equation*}
 \end{Theorem}

 \begin{proof}
By Theorem~\ref{thm:holley} it suffices to show that 
\begin{equation}
\label{eq:holley}
\mu_{X,p,q,i}\paren{P\cup P'}\mu_{X,p,q,i}\paren{P\cap P'} \geq \mu_{X,p,q,i}\paren{P} \mu_{X,p,q,i}\paren{P'}
\end{equation}
for any $P,P'.$ As 
\[\abs{P\cup P'}+\abs{P\cap P'}=\abs{P}+\abs{P'}\]
the desired statement will follow from
 \[\abs{H^{i-1}\paren{P \cap P'}}\abs{H^{i-1}\paren{P \cup P'}} \geq \abs{H^{i-1}\paren{{P}}}  \abs{H^{i-1}\paren{{P'}}}\,,\]
where cohomology is taken with coefficients in $\Z_q.$ 

 Consider the following part of the Mayer--Vietoris sequence for $P$ and $P':$
\[\begin{tikzcd}[cramped, sep = small]
H^{i-1}\paren{P\cup P'} \arrow{r}{\varphi}  & H^{i-1}\paren{P} \oplus H^{i-1}\paren{P'} \arrow{r}{\chi} & H^{i-1}\paren{P \cap P'} \arrow{r}{\delta} & H^{i}\paren{P \cup P'}
\end{tikzcd}
\,.\]
We apply the first isomorphism theorem for abelian groups (that is, that if $\phi:G\rightarrow H$ is a homomorphism then $G\cong \im \phi \oplus \ker \phi$) to the first three terms of the sequence.

First,
\[H^{i-1}\paren{P\cup P'}\cong  \im \varphi \oplus \ker \varphi\]
so 
\[\abs{H^{i-1}\paren{P\cup P'}}=\abs{ \im \varphi}\abs{ \ker \varphi}.\]

Next,
\[ H^{i-1}\paren{P} \oplus H^{i-1}\paren{P'}\cong \im \chi \oplus \ker \chi \cong \im \chi \oplus \im \varphi\]
and
\[ \abs{H^{i-1}\paren{P}}\abs{H^{i-1}\paren{P'}}=\abs{H^{i-1}\paren{P} \oplus H^{i-1}\paren{P'}}=\abs{\im \chi}\abs{\im \varphi}\,.\]
A similar computation yields
\[ \abs{H^{i-1}\paren{P\cap P'}}=\abs{\im \chi}\abs{\im \delta}\,.\]

Combining these results gives
\begin{align*}
 \abs{H^{i-1}\paren{P \cap P'}}\abs{H^{i-1}\paren{P \cup P'}}&= \abs{\im \chi}\abs{\im \delta}\abs{ \im \varphi}\abs{ \ker \varphi}\\
 &\geq \abs{\im \chi}\abs{\im \varphi}\\ 
 &= \abs{H^{i-1}\paren{{P}}} \abs{H^{i-1}\paren{{P'}}}\,,
\end{align*}
 as desired.
 \end{proof}

\begin{Corollary}\label{cor:FKGinfinite}
    The measures $\mu^{\xi}_{r,p,q,i}$ and $\mu^{\overline{\xi}}_{r,p,q,i}$ are also positively associated, as are the limiting infinite volume PRCMs, if they exist.
    \end{Corollary}
\begin{proof}
Let $\mu_r$ be a measure of the form  $\mu^{\xi}_{r,p,q,i}$ or $\mu^{\overline{\xi}}_{r,p,q,i}.$ By Proposition~\ref{prop:largeboundary} and Corollary~\ref{cor:boundarywelldefined} for a sufficiently large cube $\Lambda\supset r$ we have that
$$\mu_r\stackrel{d}{=}\paren{\mu_{\Lambda} \| A} \Big|_{r}$$
where $A$ is the event that the states of all plaquettes in $\Lambda\setminus r$ agrees with those of $\xi.$ It is easy to see that Equation~\ref{eq:holley} for $\mu_\Lambda$ implies it for $\mu_r.$  

The positive association for infinite volume measures follows from that of the finite volume measures by Proposition 4.10 of~\cite{grimmett2006random}.
\end{proof}

\newpage
\bibliographystyle{alpha}
\bibliography{bibliography}

\end{document}